\documentclass[11pt, a4paper]{article}

\usepackage{amsmath, epsfig, url}
\usepackage{latexsym}
\usepackage{amssymb}
\usepackage{latexsym}
\usepackage{graphicx}

\newtheorem{theorem}{Theorem}
\newtheorem{lemma}[theorem]{Lemma}
\newtheorem{corollary}[theorem]{Corollary}
\newtheorem{proposition}[theorem]{Proposition}

\newtheorem{problem}{Problem}

\input amssym.def
\input amssym.tex

\long\def\delete#1{}

\def\qed{\hfill$\Box$\vspace{11pt}}

\def\Xi{X}

\def\EE{{\cal E}}
\def\KK{{\cal V}}

\def\De{\Delta}

\def\d{\delta}
\def\g{\gamma}
\def\l{\lambda}

\def\diam{{\rm diam}}

\newcommand{\be}{\begin{equation}}
\newcommand{\ee}{\end{equation}}
\newcommand{\bea}{\begin{eqnarray}}
\newcommand{\eea}{\end{eqnarray}}
\newcommand{\bean}{\begin{eqnarray*}}
\newcommand{\eean}{\end{eqnarray*}}

\title{{\bf Three-arc graphs: characterization and domination}}

\author{Guangjun Xu~and~Sanming Zhou \\ \\
{\small 
Department of Mathematics and Statistics}\\
{\small The University of Melbourne}\\
{\small Parkville, VIC 3010, Australia}\\
{\small E-mail: {\texttt \{gx, smzhou\}@ms.unimelb.edu.au}}}

\date{\today}

\begin{document}

\openup 0.8\jot
\maketitle

\vspace{-0.5cm}
 
\smallskip
\begin{abstract}
 An arc of a graph is an oriented edge and a 3-arc is a 4-tuple
$(v,u,x,y)$ of vertices such that both $(v,u,x)$ and $(u,x,y)$
are paths of length two. 
The 3-arc graph of a graph $G$ is 
defined to have vertices the arcs of $G$ such that two arcs
$uv, xy$ are adjacent if and only if $(v,u,x,y)$ is a
3-arc of $G$. In this paper we give a characterization of 3-arc graphs and 
obtain sharp upper bounds on the domination number of the 3-arc graph of a graph
$G$ in terms that of $G$. 


\end{abstract}

\section{Introduction}

The $3$-arc construction \cite{Li-Praeger-Zhou98} is a relatively new graph operation that has been used in the classification or characterization of several families of arc-transitive graphs \cite{Gardiner-Praeger-Zhou99, MPZ, Li-Praeger-Zhou98, LZ, Zhou00c, Zhou98}. (A graph is arc-transitive if its automorphism group acts transitively on the set of oriented edges.) As noted in \cite{KXZ}, although this operation was first introduced in the context of graph symmetry, it is also of interest for general (not necessarily arc-transitive) graphs, and many problems on this new operation remain unexplored. In this paper we give partial solutions to two problems posed in \cite{KXZ} regarding 3-arc graphs.   

An \emph{arc} of a graph $G$ is an ordered pair of
adjacent vertices. For adjacent vertices $u, v$ of $G$, we use
$uv$ to denote the arc from $u$ to $v$, $vu$ ($\ne uv$) the arc
from $v$ to $u$, and $\{u, v\}$ the edge between $u$ and $v$.
A \emph{$3$-arc} of $G$ is a 4-tuple $(v, u, x, y)$ of vertices, possibly with $v = y$, 
such that both $(v,u,x)$ and $(u,x,y)$ are paths of $G$. 
Let $\De$ be a set of $3$-arcs of $G$. Suppose that $\De$ is {\em self-paired} in the sense that $(y, x, u, v) \in \De$ whenever $(v, u, x, y) \in \De$. Then the {\em $3$-arc graph of $G$ relative to 
$\De$}, denoted by $X(G, \De)$, is defined \cite{Li-Praeger-Zhou98} to be the (undirected) graph whose vertex set is the set of arcs of $G$ such that two vertices corresponding to arcs $uv$ and $xy$ are adjacent if and only if $(v,u,x,y) \in \De$. In the context of graph symmetry, $\De$ is usually a self-paired orbit
on the set of $3$-arcs under the action of an automorphism group of $G$. In the case where $\De$ is the set of all $3$-arcs of $G$, we call $X(G, \De)$ the {\em $3$-arc graph} \cite{KZ} of $G$ and denote it by $X(G)$. 

The first study of $3$-arc graphs of general graphs was conducted by Knor and Zhou in \cite{KZ}. Among other things they proved that if $G$ has vertex-connectivity $\kappa(G) \ge 3$ then its $3$-arc graph has vertex-connectivity $\kappa(X(G)) \ge (\kappa(G)-1)^2$, and if $G$ is connected of minimum degree $\d(G) \ge 3$ then the diameter $\diam(X(G))$ of $X(G)$ is equal to $\diam(G)$, $\diam(G) + 1$ or $\diam(G) + 2$. In \cite{bmg}, Balbuena, Garc{\'{\i}}a-V{\'a}zquez and Montejano improved the bound on the vertex-connectivity by proving $\kappa(X(G)) \ge \min\{\kappa(G)(\d(G)-1), (\d(G)-1)^2\}$ for any connected graph $G$ with $\d(G) \ge 3$. They also proved \cite{bmg} that for such a graph the edge-connectivity of $X(G)$ satisfies $\l(X(G)) \ge (\d(G)-1)^2$, and they further gave a lower bound on the restricted edge-connectivity of $X(G)$ in the case when $G$ is 2-connected. In \cite{KXZ}, Knor, Xu and Zhou studied the independence, domination and chromatic numbers of $3$-arc graphs.

In a recent paper \cite{XZ} we obtained a necessary and sufficient condition \cite[Theorem 1]{XZ} for $X(G)$ to be Hamiltonian. In particular, we proved \cite[Theorem 2]{XZ} that a $3$-arc graph is Hamiltonian if and only if it is connected, and that if $G$ is connected with $\d(G) \ge 3$ then all its iterative $3$-arc graphs $X^i(G)$ are Hamiltonian, $i \ge 1$. (The iterative $3$-arc graphs are recursively defined by $X^1(G) := X(G)$ and $X^{i+1}(G) := X(X^i(G))$ for $i \ge 1$.) As a consequence we obtained \cite[Corollary 2]{XZ} that if a vertex-transitive graph is isomorphic to the $3$-arc graph of a connected arc-transitive graph of degree at least three, then it is Hamiltonian. This provides new support to the well-known Lov\'asz-Thomassen conjecture \cite{Thom1} which asserts that all connected vertex-transitive graphs, with finitely many exceptions, are Hamiltonian. We also proved (as a consequence of a more general result) \cite[Theorem 4]{XZ} that if a graph $G$ with at least four vertices is Hamilton-connected, then so are its iterative $3$-arc graphs $X^i(G)$, $i \ge 1$. 

The $3$-arc construction was generalized to directed graphs in \cite{KXZ}. Given a directed graph $D$, the {\em $3$-arc graph} \cite{KXZ} of $D$, denoted by $X(D)$, is defined to be the undirected graph whose vertex set is the set of arcs of $D$ such that two vertices corresponding to arcs $uv, xy$ of $D$ are adjacent if and only if $v \ne x$, $y \ne u$ and $u, x$ are adjacent in $D$. Recently, we proved with Wood \cite{WXZ} that the well-known Hadwiger's graph colouring conjecture \cite{West} is true for the $3$-arc graph of any directed graph with no loops.   

In spite of the results above, compared with the well-known line graph operation \cite{HB} and the 2-path graph operation \cite{BH, LL}, our knowledge of $3$-arc graphs is quite limited and many problems on them are yet to be explored. For instance, the following problems were posed in \cite{KXZ}: 
 
\begin{problem}
Characterize 3-arc graphs of connected graphs.  
\end{problem}

\begin{problem}
Give a sharp upper bound on $\g(X(G))$ in terms of $\g(G)$ for any connected graph 
$G$ with $\d(G) \ge 2$, where $\g$ denotes the domination number.
\end{problem}

In this paper we give partial solutions to these problems. We first show that there is no forbidden subgraph characterization of $3$-arc graphs (Proposition \ref{le:parti}), and then we provide a descriptive characterization of $3$-arc graphs (Theorem \ref{le:char}). We give a sharp upper bound for 
$\g(X(G))$ in terms of $\g(G)$ (Theorem \ref{th:dom}) for any graph $G$ with $\d(G) \ge 2$, and more upper bounds for $\g(X(G))$ in terms of $\g(G)$ and the maximum degree $\De(G)$ when $2 \le \d(G) \le 4$ (Theorem \ref{th:deg34}). Finally, we prove that if $G$ is claw-free with $\d(G) \ge 2$, then $\g(X(G)) \le 4\g(G)$ and moreover this bound is sharp (Theorem \ref{th:claw}).
   
All graphs in the paper are finite and undirected with no loops or multiple edges. The {\em order} of a graph is the number of vertices in the graph. As usual, the minimum and maximum degrees of a graph $G = (V(G), E(G))$ are denoted by $\delta(G)$ and $\Delta(G)$, respectively. The degree of a vertex $v \in V(G)$ in $G$ is denoted by $\deg(v)$. The {\em neighbourhood} of $v$ in $G$, denoted by $N(v)$, is the set of vertices of $G$ adjacent to $v$, and the {\em closed neighbourhood} of $v$ is defined as $N[v] := N(v)\cup \{v\}$. We say that 
$v$ {\em dominates} every vertex in $N(v)$, or every vertex in $N(v)$ is {\em dominated} by $v$.  
For a subset $S$ of $V(G)$, denote $N(S) := \cup_{v\in S} N(v)$ and $N[S] := N(S)\cup  S$. We may add subscript $G$ to these notations (e.g. $\deg_{G}(v)$) to indicate the underlying graph when there is a risk of confusion. If $N[S]=V(G)$, then $S$ is called a {\em dominating set} of $G$. The {\em domination number} of $G$, denoted by $\gamma(G)$,  is the minimum cardinality of a dominating set of $G$; a dominating set of $G$ with cardinality $\gamma(G)$ is  called a {\em $\gamma(G)$-set} of $G$. The subgraph of $G$ induced by $S$ is denoted by $G[S]$, and the subgraph of $G$ induced by $V(G) - S$ is denoted by $G-S$. 
  
The reader is referred to \cite{West} for undefined notation and terminology.

\section{A characterization of  $3$-arc graphs}
\label{sec:char}

It is well known that line graphs can be characterized by a finite set of forbidden induced subgraphs \cite{HB}. In contrast, a similar characterization does not exist for 3-arcs graphs, as we show in the following result.

\begin{proposition}
\label{le:parti}
There is no characterization of $3$-arc graphs by a finite set of forbidden induced subgraphs. More specifically, any graph is isomorphic to an induced subgraph of some $3$-arc graph. 
\end{proposition}

\begin{proof}
Let $H$ be any graph. Define $H^*$ to be the graph obtained from $H$ by adding a new 
vertex $x$ and an edge joining $x$ and each vertex of $H$. It is not hard see that $u, v \in V(H)$ are adjacent in $H$ if and only if the arcs $ux, vx$ of $H^*$ are adjacent in $X(H^*)$. Thus the subgraph of $X(H^*)$ induced by $A := \{vx  : v \in V(H)\} \subseteq V(X(H^*))$ is isomorphic to $H$ via the bijection $v \leftrightarrow vx$ between $V(H)$ and $A$. Since $H$ is arbitrary, this means that any graph is  isomorphic to an induced subgraph of some $3$-arc graph, and so the result follows.  
\end{proof} \qed

Next we give a descriptive characterization of 3-arc graphs. To avoid triviality we assume that the graph under consideration has at least one edge.  

\begin{theorem}
\label{le:char}
 A  graph  $G$ having at least one edge is isomorphic to the $3$-arc graph of some graph if and only if $V(G)$ admits a partition $\KK:= \KK_{1} \cup \KK_{2}$ and $E(G)$ admits a partition $\EE$ such that the following hold: 
    
    \begin{itemize}
\item[\rm (a)]
    each element of $\KK_{1}$ contains exactly one vertex of $G$, and each element of $\KK_{2}$ is an independent set of $G$ with at least two vertices;
\item[\rm (b)] 
each $E_{i} \in \EE$ induces a complete bipartite subgraph $B_i$ of $G$ with each part of the bipartition a subset of some $V \in \KK_{2}$ with size $|V|-1$; 
\item[\rm (c)]
if $v\in V \in \KK_{2}$, then $v$ belongs to at most $|V|-1$ complete bipartite graphs described in (b);
\item[\rm (d)]
   if two distinct complete bipartite graphs $B_i$ and  $B_j$ in (b) have parts contained in the same  $V \in \KK_{2}$, then
    $B_i$  and  $B_j$ have exactly  $|V|-2$   common vertices, and all of them are in $V$;
\item[\rm (e)]  
          $2 |\EE| = \sum_{V \in \KK_{2}}|V| - |\KK_{1}|$.
\end{itemize}
\end{theorem}

\begin{proof}
For a graph $H$ and a vertex $v$ of $H$, denote by $A_H(v)$ the set of arcs of $H$ with tail $v$.

\medskip 
\textsf{Necessity:}~ Suppose that $G$ is isomorphic to $X(H)$ for some graph $H$. We  identify $G$ with $X(H)$. Let 
$\KK_1:=\{A_H(v): \deg_H(v)=1\}$ and $\KK_2:=\{A_H(v): \deg_H(v)\ge2 \}$.
Then each element of  $\KK_{1}\cup  \KK_{2}$ is an independent set of $G$, and each edge
of $G$ occurs only between distinct elements $A_H(u)$, $A_H(v)$ of  $\KK_{2}$ with
$u, v$ adjacent in $H$. For each pair of adjacent vertices $u,v$ of $H$ with $\deg_H(u) \ge 2$ and $\deg_H(v)\ge 2$, the set $E_{\{u,v\}}$ of edges of $G$ between $A_H(u)$  and  $A_H(v)$ induces a complete bipartite subgraph of $G$. Denote by $\EE$ the family 
of such $E_{\{u,v\}}$. It is straightforward to verify that $\KK := \KK_1 \cup \KK_2$ is a partition of $V(G)$ and $\EE$ is a partition of $E(G)$ such that (a)-(e) are satisfied.

\medskip
\textsf{Sufficiency:}~ Suppose that $V(G)$ admits a partition $\KK:= \KK_{1} \cup \KK_{2}$ and $E(G)$ admits a partition $\EE$ satisfying (a)-(e). We construct a graph $H$ such that $X(H)$ is isomorphic to $G$.

We construct for each $V_x \in \KK$ a vertex $x$ of $H$. We say that $x$ represents $V_x$, and that $x$ is of type $\KK_1$ or $\KK_2$ according to whether $V_x$ belongs to $\KK_1$ or $\KK_2$. For each bipartite graph $B_i$ as in (b), there are distinct elements $V_x, V_y \in \KK_{2}$ and vertices $v_{i,x} \in V_x, v_{i,y} \in V_y$ such that $\{V_x - \{v_{i,x}\}, V_y - \{v_{i,y}\}\}$ is the bipartition of $B_i$. The pair $\{x, y\}$ is determined uniquely by $B_i$, and vice versa (so we may write $i=i(x,y)$), and we add the edge $\{x, y\}$ to $H$. For each $V_x \in \KK_{2}$, denote by $b_x$ the number of complete bipartite graphs $B_i$ as in (b) that contain at least one vertex of $V_x$. By (b), one part of the bipartition of each of such graphs $B_i$ must be a subset of $V_x$ with size $|V_x| - 1$. On the other hand, by (c) each $v\in V_x$ belongs to at most $|V_x|-1$  such complete  bipartite graphs $B_i$. By counting the number of ordered pairs $(v, B_i)$ with $v\in V_x \cap V(B_i)$, we obtain $b_x (|V_x|-1) \leq|V_x|(|V_x|-1)$, yielding $b_x \leq |V_x|$. We then add $|V_x|-b_x$ edges to $H$ joining $x$ to $|V_x|-b_x$ vertices of type $\KK_{1}$,  in such a way that no vertex of type $\KK_{1}$ is repeatedly used. Thus each vertex $x$ of $H$ of type $\KK_{2}$  has degree $|V_x|$ in $H$.  Note that the sum $\cup_{V_x \in \KK_{2}}|V_x|$ counts each edge  between two vertices of type $\KK_{2}$ twice, and 
each edge with one end-vertex of type $\KK_1$ once. The total number of vertices of type $\KK_1$ required is $\cup_{V_x \in \KK_{2}}|V_x|-2|\EE|$, which agrees with $|\KK_{1}|$ by (e). This completes the construction of $H$.
  
We now prove that $X(H)$ is isomorphic to $G$. By the construction above, each vertex $x$ of $H$ has degree $|V_x|$ in $H$, and the set $A_{H}(x)$ of arcs of $H$ outgoing from $x$ is an independent set of $X(H)$ with size $|V_x|$. Obviously, such independent sets $A_{H}(x)$ of $X(H)$ are in one-to-one correspondence with the elements $V_x$ of $\KK$. Note that $\{A_{H}(x): x \in V(H)\}$ is a partition of the vertex set $A(H)$ of $X(H)$ which corresponds to the partition $\KK = \{V_x: x \in V(H)\}$ of the vertex set of $G$.  
  
Let $\{x,y\}$ be an edge of  $H$  with $\deg_H(x)\ge2$ and $\deg_H(y)\geq2$. Then $A_H(x)\cup A_H(y)$ induces a complete bipartite subgraph $B(x,y)$ of $X(H)$ with bipartition $\{A_H(x)-\{xy\}, A_H(y)-\{yx\}\}$. On the other hand, by the definition of $H$, the edges of $G$ between $V_x - \{v_{i,x}\}$ and $V_y - \{v_{i,y}\}$ induce a complete bipartite subgraph of $G$ that is equal to $B_{i}$  as in (b) with $i = i(x, y)$. It can be verified that $\{x, y\} \mapsto B_{i(x,y)}$ defines a bijection from the set of edges $\{x,y\}$ of $H$ with $\deg_H(x)\ge2$ and $\deg_H(y)\geq2$ to the set of complete bipartite graphs as in (b).    
  
For a fixed $V_x \in \KK_{2}$, the corresponding vertex $x$ has degree at least $2$ in $H$. For each neighbour $y$ of $x$ in $H$ with $V_y \in \KK_{2}$, the complete bipartite graph $B_{i}$ with $i = i(x,y)$ has bipartition $\{V_x - \{v_{i,x}\}, V_y - \{v_{i,y}\}\}$. By the construction of $H$, one can verify that the mapping $xy \mapsto v_{i,x}$ (where $i = i(x,y)$) is a bijection from $\{xy: y \in N_H(x), \deg_H(y)\geq2\}$ (which is a subset of $A_H(x)$) to $V_x$. Let $L := \{y \in N_H(x): \deg_H(y) = 1\}$ be the set of  leaf-neighbours of  $x$, and $W_x$ be the  set of vertices $w \in V_x$ such that there exists no $B_i$ as in (b) with $V(B_i) \cap V_x = V_x-\{w\}$. Then  $|L|=|W_x|$ and so we can choose a bijection (in an arbitrary manner) between  $L$  and $W_x$. Finally, we choose an arbitrary bijection between the set of arcs of $H$ starting from leaves and $\KK_{1}$. Then we have defined a bijection between the vertices of $X(H)$ and the vertices of $G$. From the way this bijection is defined it is straightforward to verify that it is an isomorphism between $X(H)$ and $G$.   
\qed
\end{proof}

\section{Domination number of $3$-arc graphs}

\delete
{
In this section we prove an upper bound on $\g(X(G))$ in terms of $\g(G)$ for any graph $G$ with $\d(G) \ge 2$ (Theorem \ref{th:dom}), three upper bounds on $\g(X(G))$ in terms of $\g(G)$ and $\De(G)$ when $G$ has $\d(G)=2,3,4$ respectively (Theorem \ref{th:deg34}), and $\g(X(G)) \le 4\g(G)$ for claw-free graphs $G$ with $\d(G) \ge 2$ (Theorem \ref{th:claw}).
}

Given a graph $G$, denote by $G \circ K_1$ the graph obtained from $G$ by adding for each $x \in V(G)$ a new vertex $x'$ and a new edge joining $x$ and $x'$. Define $\cal{A}$ to be the family of graphs depicted in Figure \ref{fig:a}. 

\vspace{0.5cm} 
  
\begin{figure}[htb]
\begin{center}
\includegraphics[width=9cm]{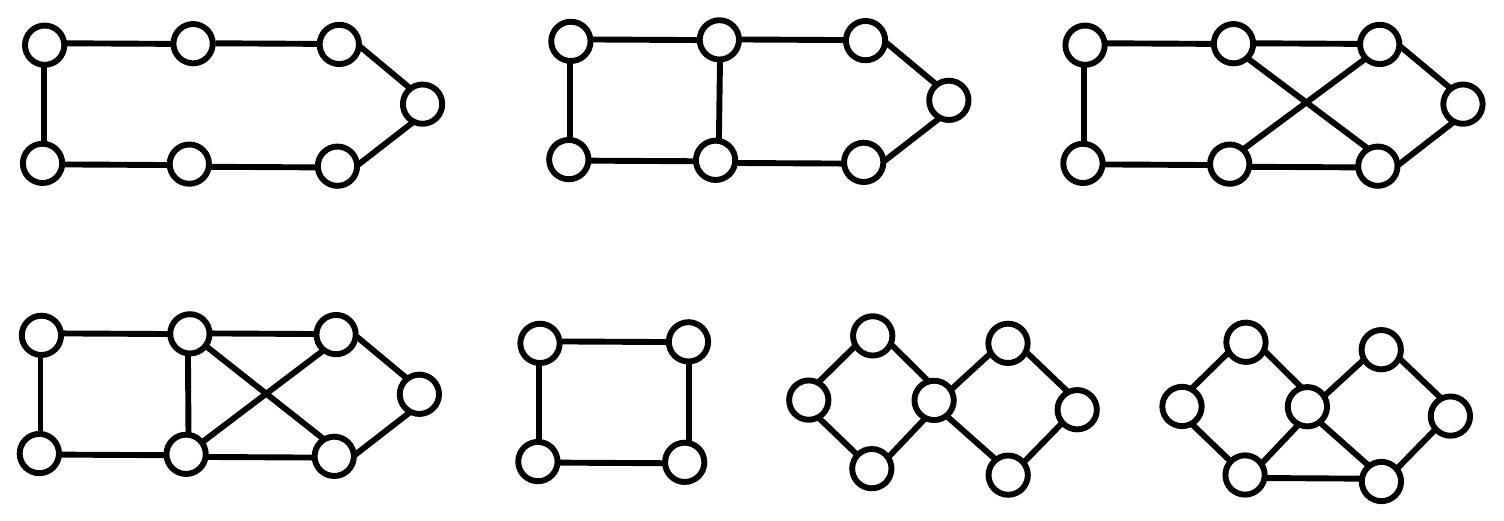}
\caption{Graphs in family $\cal{A}$.}
\label{fig:a}
\end{center}
\end{figure}

\vspace{-0.5cm}
 
 \begin{lemma}
 \label{le0}
 Let $G$ be a connected graph of order $n$.   
 \begin{itemize} 
\item[\rm (a)] If $\delta(G)\geq 1$, then $\gamma(G)\leq  \displaystyle{\frac{n}{2}}$ {\em(\cite[p.206]{O})}, and $\gamma(G)= \displaystyle{\frac{n}{2}}$  if and only if $G\cong C_4$ or $H  \circ K_1$ for some graph $H$   {\em(\cite{FJK,PX})};  
\item[\rm (b)]  
  if $\delta(G)\geq 2$  and $G\notin \cal{A}$, then  $\gamma (G)\le   \displaystyle{\frac{2n}{5}}$ {\em(\cite{MS})};
\item[\rm (c)]  
  if $\delta(G)\geq 3$, then $\gamma (G)\le   \displaystyle{\frac{3n}{8}}$ {\em(\cite{RE})}.
\end{itemize}
 \end{lemma}
 
Given a graph $G$, define
\be
\label{eq:vi}
V_i (G) := \{x \in V(G): \deg(x) \ge i\},\;\, i \ge 0.
\ee
For a fixed subset $U$ of $V(G)$, a subset $D\subseteq V(G)$  
  is called  a  {\em $(G:U)$-dominating set}     if   $U \subseteq N[D]$. The {\em $(G:U)$-domination number},  denoted by   $\gamma(G:U)$,  is the minimum
  cardinality of a $(G:U)$-dominating set. 
   Note that   a $(G:U)$-dominating set needs not
    be a subset of $U$,  and 
   a $(G: V(G))$-dominating set is precisely  an ordinary  dominating set of  $G$. 
  
 \begin{lemma}
 \label{le1}
Let $G$ be a graph with order  $n$.  Then 
 \begin{itemize} 
\item[\rm (a)]   
    $\gamma (G:V_1(G))\le   \displaystyle{\frac{n}{2}}$, and  equality holds if and only if each component of $G$ is isomorphic to  $C_4$ or $H \circ K_1$ for some connected graph $H$ (which relies on the component);
\item[\rm (b)]
 $\gamma (G:V_2(G))\le   \displaystyle{\frac{2n}{5}}$ if each component of 
    $G$ is not isomorphic to a graph in the family $\cal{A}$;
\item[\rm (c)]  
   $\gamma (G:V_3(G))\le   \displaystyle{\frac{3(n+2)}{8}}$.
\end{itemize}
 \end{lemma}

 \begin{proof}  Let $G'$ be the graph obtained from $G$ by deleting all isolated vertices.
 Then $|V(G')| \leq n$ and    $\gamma (G:V_i(G))\le  \gamma(G')$ for each $i\geq1$.
  
  (a) By  Lemma \ref{le0}(a),    $\gamma (G:V_1(G))\le  \gamma(G') \leq |V(G')|/2 \le n/2$.
Thus $\gamma (G:V_1(G))  =  n/2$ if and only if all equalities throughout this inequality chain
    hold. The  first equality in the chain holds if and only if  $G$ contains no isolated vertex, and 
  the second equality holds if and only if  (see Lemma \ref{le0}(a)) each component of 
    $G'$ is isomorphic to $C_4$ or $H \circ K_1$ for some connected graph $H$.  
 Hence the result in (a) follows.

(b) Denote $G_2:= G[V_{2}(G)]$ and $W:= N(V_{2}(G)) - V_{2}(G)$. 
Then $W$ contains all vertices of $G$ outside $V_{2}(G)$ that have exactly one
neighbour in $V_{2}(G)$. Thus $G[W]$ consists of isolated vertices, say, 
$x_1,x_2,\ldots, x_l$, where $l\geq 0$. Denote by $x_i'$ the unique neighbour of $x_i\in W$  in $V_{2}(G)$. Note that it may happen that $x_i'=x_j'$ for distinct  $x_i,  x_j \in W$.

 Let $G_2 := G[V_{2}(G) \cup W]$. 
Then  each vertex  in $V_{2}(G)$  has degree at least 2 in  $G_2$.  Define a new graph   $J$ as follows.
   If $l=0$, set  $J:  =G_2$;  
     if $l\geq 2$,  let  $J$ be the graph obtained 
 from   $G_2$ by adding all possible edges to  $G[W]$ until it becomes a path  of length $l-1$;
 if  $l=1$,  let  $J$ be the graph obtained 
 from   $G_2$ by    joining $x_{1}$ to  a neighbour of $x_1'$ other than $x_1$  in $G_2$
 (such a neighbour exists because $\deg_{G_2}(x_1')\geq 2$).

 It is easy to see that $J$ has minimum degree 2.  Let $R_1,R_2,\ldots, R_s$ be the components of
   $J$.  Let $r_j$ be the order of $R_j$, and let $T_j:= V_2(R_j) \cap V_{2}(G)$, 
		for $1\le j\le s$.  Then  $\gamma (R_j: T_j) \leq  	\gamma (R_j: V_2(R_j)) =  \gamma (R_j)$.
		We claim that 
		$\gamma (R_j: T_j) \leq 2r_j/5$. In fact,  if $R_j \notin \cal{A}$, then  by  Lemma \ref{le0}(b)    we have $\gamma (R_j: T_j) \leq  \gamma (R_j) \leq 2r_j/5$.
		Suppose that $R_j \in \cal{A}$. Since no component of $G$ is isomorphic to a graph in  $\cal{A}$, by examining each graph in $\cal{A}$, we see that at least  one edge joining two degree-two vertices  in  $R_j$ is not in $G$
		(such an edge was introduced in the construction of $J$). Thus at least one vertex of  $R_j$ is not in $T_j$. 
   It is readily   seen that  $\gamma (R_j: T_j) \le 1 < 8/5 = 2r_j/5$ if  $R_j\cong C_4 \in \cal{A}$,
	and  $\gamma (R_j: T_j) \le 2 < 14/5 = 2r_j/5$ otherwise. 
Since the union of all $\gamma (R_j: T_j)$-sets is a 
 $(G_2: V_{2}(G))$-dominating set, it follows that  
  $\gamma (G:V_2(G))\le 2|V(G_2)|/5 \leq 2n/5$.

	\delete{
	If    $l= 1$, without loss of generality assume that $x_1$ is in    $R_1$. Denote by  $R_1'$   the component 
    containing $x_1$  in  $G_2$.  Then  $R_1'$   is exactly the graph obtained from $R_1$ by removing the edge
     $\{x_{1},  z\}$.  Clearly  $\gamma (R_j) \leq  \displaystyle{\frac{2r_j}{5}}$ for all $j\ne 1$.
      We claim that   $\gamma(R_1': V_2(R_1')) \leq \gamma (R_1)$.
  Let $T$ be a $\gamma (R_1)$-set. If $x_1\notin T$, it is clear that $T$ is a $(R_1': V_2(R_1'))$-dominating set
  by noting that $x_1$ is of degree 1 in $R_1'$;   if $x_1\in T$, then  $(T-\{x_1\})  \cup \{x_1'\}$ is  a $(R_1': V_2(R_1'))$-dominating set.   Thus,  $\gamma(R_1': V_2(R_1')) \leq \gamma (R_1)$.  Similarly, we have  $\gamma (G:V_2(G))\le  \displaystyle{\frac{2|V(G_2)|}{5}} \leq \displaystyle{\frac{2n}{5}}$. 
 }

    (c)     Denote $G_3:=  G[V_{3}(G)]$ and $W := N(V_{3}(G)) - V_{3}(G)$. Since each vertex of $W$ has at least one neighbour in $V_3(G)$, if a vertex of $W$ has degree two in $G[W]$, then it must be in $V_{3}(G)$ and so cannot be in $W$, a contradiction. Therefore, $\De(G[W]) = 1$ and $G[W]$ consists of isolated vertices and independent edges. Denote $W = \{x_1,x_2,\ldots, x_l\}$,  where $l\geq 0$, and for each $i$ choose $x_i'$ to be any (but fixed) neighbour of $x_i$ in $V_{3}(G)$.

 Let $G_3 := G[V_{3}(G) \cup W]$.  
 Define a new graph   $J$ as follows:
   If $l=0$, set  $J:  =G_3$;  
     if $l\geq 3$,  let  $J$ be the graph obtained 
 from   $G_3$ by adding all possible edges to  $G[W]$ until it becomes a cycle of length $l$;
 if  $l=1$,  let  $J$ be the graph obtained 
 from   $G_3$ by adding  two new vertices   $u_1,u_2$ together with   edges   
    $\{u_{1},  u_2\}$,   $\{x_{1},  u_1\}$,   $\{x_{1},  u_2\}$,    $\{u_1,  x_1'  \}$  and  $\{u_2,  x_1'\}$;  if  $l=2$,  let  $J$ be the graph obtained 
 from   $G_3$ by adding  two new vertices   $u_1,u_2$ together with all  possible edges 
 among $u_{1},  u_2,  x_{1}$ and $x_2$.
 
Note that $\d(J) = 3$. By Lemma \ref{le0}(c), we have  $\gamma (J)\leq  3|V(J)|/8$. Let $D$ be a $\gamma(J)$-set. If $D\cap \{u_1,u_2\}=\emptyset$, then $D$ is a dominating set of  $G_3$ and so $\gamma(G_3)\leq \gamma(J)$.
 If $D\cap \{u_1,u_2\}\ne \emptyset$,   then by the minimality  of $D$ we have  
 $D\cap \{x_1,  x_2\}=\emptyset$  and 
$D$ contains exactly one of  $u_1$ and $u_2$. Thus 
 $(D-\{u_1,u_2\}) \cup \{x_1\}$ is a dominating set of  $G_3$,   and  again we have  $\gamma(G_3)\leq \gamma(J)$. Since any dominating set of  $G_3$ is also a  $(G_3: V_{3}(G))$-dominating set, we obtain that $\gamma (G:V_3(G)) \le \g(G_3) \le  3|V(J)|/8 \leq 3(|V(G_3)|+2)/8  \leq 3(n+2)/8$.  
 \qed
 \end{proof}

\begin{theorem}
\label{th:dom}
Let $G$ be a graph with minimum degree $\delta := \d(G) \geq2$. Then 
\be
\label{eq:sq}
   \gamma (X(G))\leq 3\gamma(G) + \min_{S\in {\cal Q}}\ \{\gamma (G_S: V_{\delta-1}(G_S))\}-1,
\ee
   where $G_S=G - S$ and ${\cal Q}$ is the set of $\gamma(G)$-sets of $G$. Moreover, this bound is attainable.
\end{theorem}

\begin{proof}
  Let $S$ be a $\gamma(G)$-set of $G$ and denote $H := G_S$ for simplicity. As in (\ref{eq:vi}), let $V_{j}(H)$ denote the set of vertices of 
  $V(H)$  that have degree at least $j$ in $H$. With each $x\in S$ we associate a set 
\be
\label{eq:ax}
A(x) := \{xx_1, xx_2, x_2x_3\}
\ee
 of three arcs of $G$, where $x_1, x_2$  are   distinct neighbours of $x$, and $x_3$ is a neighbour of $x_2$ other than $x$. Since $\delta \geq2$, such a set $A(x)$ exists for every $x\in S$. (In general, many sets $A(x)$ may be obtained this way. We choose one of them arbitrarily and fix it.) Define
\be
\label{eq:as}
      A(S) := \cup_{x\in S} A(x)
\ee
so that $|A(S)| = 3|S| = 3\gamma(G)$. Set
\be
\label{eq:W}
W := \{w\in V(G)-S: |N(w)\cap S|=1\}
\ee
$$
U:=V(G)-(S \cup W).
$$ 
Then  $S$, $W$ and $U$ form a partition of $V(G)$.  
   Since  each $w\in W$ has exactly one neighbour in $S$, it has at least $\delta-1$ neighbours in $H$. Therefore,  
    \be
    \label{eq:WV}
    W\subseteq V_{\delta-1}(H).
    \ee
Thus every $(H: V_{\delta-1}(H))$-dominating set is also an $(H: W)$-dominating set, and therefore
\be
\label{eq:WV1}
\gamma (H: W) \le \gamma (H: V_{\delta-1}(H)).
\ee

     Let $D$ be a minimum $(H: W)$-dominating set in $H$. 
      With each vertex $y \in D$  we choose an arc $yy'$  such that $y'\in N(y) \cap S$.    Such a vertex $y'$ exists  for every $y \in D$  because $y$ is dominated by a vertex in $S$. Define 
    \be
    \label{eq:D}
    A(D):=\{yy': y\in D\}.
    \ee 
   Then $|A(D)|=|D|=  \gamma (H: W)$, and so $|A(D)| \le \gamma (H: V_{\delta-1}(H))$ by (\ref{eq:WV1}).  

\medskip    
\textsf{Claim 1:} $A(S)\cup A(D)$ is a dominating set of $X(G)$. Hence
\be
\label{eq:chain}
\gamma (X(G)) \le |A(S)\cup A(D)| \le |A(S)| + |A(D)| \leq  3 \gamma(G)+ \gamma (H: V_{\delta-1}(H)).
\ee
 
\textit{Proof of Claim 1:} 
It suffices to show that 
      each arc with tail in   $S$, $W$ or  $U$   is dominated in $X(G)$ by at least one element of  $A(S)\cup A(D)$.
       Let $ab$ be an arc of $G$.  When $a\in S$, if $b=a_1$ or $a_2$ (where $a_1$ and $a_2$ are the vertices in the definition of $A(a)$), then clearly $ab \in A(S)$; if $b$ is a neighbour of $a$ other than  $a_1$ and $a_2$, then $ab$ is dominated by $a_2a_2'\in A(S)$. If $a\in U$, then $a$ has at least two neighbours  in $S$ and so at least one of them, say, $z\in S$ is different from $b$. Thus $ab$ is dominated by $zv \in A(S)$, where $v \in \{z_1,z_2\}-\{a\}$ with $z_1, z_2$ the neighbours of $z$ used in the definition of $A(z)$. Suppose that $a\in W$ and let $z$ be the unique neighbour of $a$ in $S$.
          If $b\notin S$,  then $ab$ is dominated by  $zv \in A(S)$, where $v \in \{z_1,z_2\}-\{a\}$. So we assume $b\in S$, so that $b=z$.  If $a\in D$, then $ab\in A(D)$ by the definition of $A(D)$. If $a\notin D$, then $a$ is dominated by some vertex $y\in D$ as $D$ dominates all vertices of $W$ in $H$ including $a$. Hence $ab$ is dominated by $yy'\in A(D)$ in $X(G)$. This completes the proof of Claim 1. 
					
In what follows we will show that the upper bound in (\ref{eq:chain}) can be decreased by one. In fact, in the case when $A(S) \cap A(D) \ne \emptyset$, we have $|A(S)\cup A(D)| \le |A(S)| + |A(D)| - 1$ and similar to (\ref{eq:chain}) we obtain $\gamma (X(G)) \le 3 \gamma(G)+ \gamma (H: V_{\delta-1}(H)) - 1$. We now prove that the same bound holds when $A(S) \cap A(D) = \emptyset$. This is achieved by proving:

\medskip    
\textsf{Claim 2:} If $A(S) \cap A(D) = \emptyset$, then we can modify $A(S) \cup A(D)$ to obtain a new dominating set $A_1(S) \cup A_1(D)$ of $X(G)$ with size at most $|A(S) \cup A(D)|-1$.  

\medskip
\textit{Proof of Claim 2:} We first deal with the case where $|S|=1$ (that is, $\gamma(G)=1$).  
	In this case, $D \ne \emptyset$ and $W=V(G)-S$ (and so $U = \emptyset$). Denote $S=\{x\}$ and let $y\in D$.  Since $\delta\ge 2$, we may choose a neighbour $z\ne x$ of $y$. Let $A_1(S) := \{xz, zy, yx\}$.  Note that $A(D)$ contains an arc $yy'$ with $y'\in N(y) \cap S$.
 		If	$yy'\ne yx$, then let $A_1(D)$ be obtained from $A(D)$ by replacing this particular arc $yy'$ by $yx$ but retaining all other $uu' \in A(D)$ in (\ref{eq:D}); otherwise, let $A_1(D) := A(D)$. Since $yx$ appears in both $A_1(S)$ and $A_1(D)$, we have $|A_1(S) \cup A_1(D)| = |A(S) \cup A(D)| - 1$. We claim 
		that each arc of $G$ is dominated by $A_1(S) \cup A_1(D)$. In fact, let $ab$ be an arbitrary arc of $G$. Since $W=V(G)-S$, either $a \in S$ or $a \in W$. If $a=x\in S$, then $ab$ is dominated by 
		$zy$ or $ab=xz$.  Suppose that $a\in W$. If $b=x\in S$, then $ab\in A_1(D)$ or  $ab$ is dominated by an arc in $A_1(D)$. If $b \notin S$, then $ab=zy$ or $ab$ is dominated by either $xz$ or $yx$. Therefore, $A_1(S) \cup A_1(D)$ is a dominating set of $X(G)$. 
		 
	In the rest proof of Claim 2, we assume that $|S|\ge 2$. 

\medskip    
\textsf{Case 1:}	$S$ is not an independent set of $G$. That is,  there is an edge joining two vertices, say $x$ and $y$, of $S$. 
		Let $x'\ne y$ be a neighbour of $x$ and  $y'\ne x$ be a neighbour of $y$. Let $A_1(S)$ be obtained from $A(S)$ by replacing $A(x)$ by $\{xx',xy, yy'\}$,  
		 $A(y)$ by $\{yy',yx, xx'\}$, but leaving other $A(u)$ for $u \in S - \{x, y\}$ unchanged, in (\ref{eq:as}). Since $\{xx',xy, yy'\}$ and $\{yy',yx, xx'\}$ have two common arcs, we have $|A_1(S)| = |A(S)| - 2$. Set $A_1(D) : = A(D)$. Then $|A_1(S) \cup A_1(D)| = |A(S) \cup A(D)| - 2$ and one can show that $A_1(S) \cup A_1(D)$ is a dominating set of $X(G)$. 
		
\medskip    
\textsf{Case 2:} $S$ is an independent set of $G$ and $U\ne \emptyset$. Let $z\in U$ and let $x,y\in S$ be  distinct neighbours of $z$. 
		If $\deg(z)\ge 3$, let $z'$ be a neighbour of $z$ other than $x$ and $y$. Since $\delta\ge 2$, we may choose a neighbour $x'\ne z$ of $x$ and a neighbour $y'\ne z$ of $y$. Let $A_1(S)$ be obtained from $A(S)$ by replacing $A(x)$ by $\{xx',xz, zz'\}$, $A(y)$ by $\{yy',yz, zz'\}$, but leaving other $A(u)$ with $u \in S - \{x, y\}$ unchanged, in (\ref{eq:as}). Since $\{xx',xz, zz'\}$ and $\{yy',yz, zz'\}$ have one arc in common, we have $|A_1(S)| = |A(S)| - 1$. If  $\deg(z)=2$, then let $A_1(S)$ be obtained from $A(S)$ by replacing $A(x)$ by $\{xz, zy\}$, $A(y)$ by $\{yz, zx\}$, but leaving other $A(u)$ with $u \in S - \{x, y\}$ unchanged, in (\ref{eq:as}). Then $|A_1(S)| = |A(S)| - 2$. Set $A_1(D) : = A(D)$ regardless of the degree of $z$. Then $|A_1(S) \cup A_1(D)| \le |A(S) \cup A(D)| - 1$ and $A_1(S) \cup A_1(D)$ is a dominating set of  $X(G)$.

\medskip    
\textsf{Case 3:}	$S$ is an independent set of $G$ and $U= \emptyset$.  Then $W=V(G)-S\ne \emptyset$,  
$|A(D)|\ge 1$ and $|W-D|\ge 1$. Choose a vertex $z\in W-D$. Since $S$ is a dominating set of $G$, we may choose a neighbour $x\in S$ of $z$. Similarly, we may choose a neighbour $v \in D$ of $z$ in $H$ and a neighbour $u\in S$ of $v$ in $G$. It may happen that $u=x$, but this will not affect our subsequent proof.

Let $A_1(S)$ be obtained from $A(S)$ by replacing $A(x)$ by $\{xz, zv, vu\}$,   but leaving other $A(y)$ with $y \in S - \{x\}$ unchanged, in (\ref{eq:as}). Since $v\in D$, $A(D)$ contains an arc $vv'$ with $v'\in N(v) \cap S$.
 	If	$vv'\ne vu$, then let $A_1(D)$ be obtained from $A(D)$ by replacing this particular arc $vv'$ by $vu$ but
 	 retaining all other $yy' \in A(D)$ in (\ref{eq:D}); otherwise, let $A_1(D) := A(D)$. 
Since $vu$   appears in both $A_1(D)$  and  $A_1(S)$, we have $|A_1(S) \cup A_1(D)| = |A(S) \cup A(D)| - 1$. 
We now show that $A_1(S) \cup A_1(D)$ is a dominating set of $X(G)$. 
				In fact, by the definition of $A_1(D)$  and  $A_1(S)$ one can see that every arc with tail not in $N[x]$ that is dominated by $A(S) \cup A(D)$ is now dominated by $A_1(S) \cup A_1(D)$.  
				Let $ab$ be an arc with $a\in N[x]$.  If  $a=x$, then $ab=xz\in A_1(S)$ or 
					$ab$ is dominated by $zv\in A_1(S)$.  Suppose that    $a\in N(x)$.   
				If $a=z$, then $ab=zv$ or   	$ab$ is dominated by      $vu\in A_1(S)$. 
			Suppose that  $a\in N(x)-\{z\}$. If $b\ne x$, then 		$ab$ is dominated by $xz\in A_1(S)$.
			If $b=x$, then either $ab\in A_1(D)$  when $a\in D$, or  	$ab$ is dominated by $yy'\in A_1(D)$, where $y$ is a vertex in $D$ that dominates $a$ in $H$.
			 Therefore, $A_1(S) \cup A_1(D)$ is a dominating set of  $X(G)$. 
				
So far we have completed the proof of Claim 2. By this claim (and the discussion before it) and (\ref{eq:chain}),  we obtain $\gamma (X(G)) \le |A_1(S) \cup A_1(D)| \le |A(S) \cup A(D)| - 1 \le 3 \gamma(G)+ \gamma (H: V_{\delta-1}(H)) - 1$. Since this holds for any $\g(G)$-set $S$ of $G$ and since $H = G_S$, we obtain (\ref{eq:sq}) immediately.   
			
			It is easy to see that the bound in (\ref{eq:sq}) is achieved by the $3$-cycle $C_3$ and in general by any friendship graph (that is, a graph obtained from a number of copies of $C_3$ by identifying one vertex from each copy to form a single vertex). 		
\qed
\end{proof}

It was proved in \cite{KXZ} that, for any connected graph $G$ of order $n \ge 4$ and minimum degree at least $2$, we have $\gamma(X(G))\le n$. Combining this with $n/(1+\Delta(G)) \le \gamma (G)$ \cite{berge1973}, we then have
\be
\label{eq:del}
 \gamma(X(G))   \le  (1+\Delta(G))\gamma(G).
\ee
The following theorem improves this bound for any graph $G$ with $\Delta(G) \ge \d(G) = 2, 3$ or $4$. In Theorem \ref{th:claw} we will give a further improved bound for any claw-free graph with minimum degree at least two.  
 
\delete{ 
\begin{figure}[htb]
\begin{center}
\includegraphics[width=7cm]{familyF.pdf}
\caption{A graph of   ${\cal F}_{3,4}$.}
\label{fig:f}
\end{center}
\end{figure}

 We construct a family of graphs which will be needed in the next result.
 Let $i\geq1$, $j\geq 3$ be integers such that $ij$ is even. Let ${\cal F}_{i,j}$ contain all graphs $F$ with 
 vertex set $A \cup B$  such that  $F[A]\cong \overline{K}_i$  is an independent set of   $i$ vertices,
  $F[B]\cong \frac{ij}{2} \cdot K_2$ is a set independent edges, and each vertex of $A$ is connected to 
  exactly  $j$  pairwise nonadjacent vertices of $B$ with the condition that  distinct vertices of 
  $A$ does not share a  common neighbour in $B$. Note that each vertex of  $F$  has degree  $2$ or $j$ and 
  $F$ is not necessarily connected.
  A graph in the family ${\cal F}_{3,4}$  is depicted in Figure \ref{fig:f}.
 }

\begin{theorem}
\label{th:deg34}
 Let $G$ be a graph. 
\begin{itemize} 
\item[\rm (a)] If  $\delta(G) =2$, then
    $\gamma (X(G))\le   \left(\displaystyle{\frac{\Delta(G)}{2}}+3\right)\gamma(G)-1$;
\item[\rm (b)]   if  $\delta(G) = 3$, then
    $\gamma (X(G))\le   \left(\displaystyle{\frac{2\Delta(G)}{5}}+3\right)\gamma(G)-1$;
\item[\rm (c)]   if  $\delta(G) = 4$, then
    $\gamma (X(G))\le   \left(\displaystyle{\frac{3(\Delta(G)+2)}{8}}+3\right)\gamma(G)-1$.
\end{itemize}
Moreover, the bounds in (a) and (b) are attainable. 
\end{theorem}

\begin{proof}  
Denote $\d := \d(G)$. As in the proof of Theorem \ref{th:dom}, let $S$ be a $\g(G)$-set of $G$ and denote $H := G - S$. Let $H'$ be obtained from $H$ by deleting all isolated vertices.
Then  $\gamma (H: V_{1}(H)) \leq \gamma(H')$. 

(a) If  $\delta=2$, then  by  Theorems \ref{th:dom},    Lemma \ref{le1}(a) and the fact that  $|V(G)| \leq (\Delta(G)+1)\gamma (G)$,   we have 
 \begin{eqnarray}
 \gamma (X(G)) &\leq&  3\gamma(G)+ \gamma (H: V_{1}(H))-1    \nonumber \\
                 &\leq&    3\gamma(G)+    \displaystyle{\frac{n-\gamma(G)}{2}} -1    \nonumber \\
           &\leq&    3\gamma(G)+    \displaystyle{\frac{(\Delta(G)+1)\gamma (G)-\gamma(G)}{2}}-1   \nonumber  \\  
                       &=&  \left(\displaystyle{\frac{\Delta(G)}{2}}+3\right)\gamma(G)-1. \nonumber   \label{eq:delta2} 
\end{eqnarray}

(b) Assume $\delta=3$. Let $A(S)$ and $W$ be defined by (\ref{eq:as}) and (\ref{eq:W}), respectively, with $A(x)$ as given in (\ref{eq:ax}) for each $x \in S$. By (\ref{eq:WV}), we have $W\subseteq V_{2}(H)$.
   
   Let   $R_1,R_2,\ldots, R_s$ be the set of  components of
   $H'$, and let $r_j$ be the order of $R_j$. We are going to prove that we can choose an 
   appropriate subset $D_j$ of $V(R_j)$ for each $j$ such that $|D_j| \leq 2r_j/5$ and 
   $A(S)\cup A(\cup_{j=1}^s D_j)$ is a dominating set of $X(G)$.  
   
   In fact, if $R_j$ is not 
   isomorphic to any graph in the family $\cal{A}$ (see Figure \ref{fig:a}), then we choose $D_j$ to be a 
   minimum $(R_j: V_2(R_j))$-dominating set. By Lemma \ref{le1}(b), we have 
   $|D_j| = \gamma(R_j: V_2(R_j)) \leq 2r_j/5$.

   Suppose that a component $R_j$ is isomorphic to  some graph in $\cal{A}$.   If  $R_j$ contains a vertex $z$  which  is not in $W$, 
    then we choose  $D_j$ to be a minimum    dominating set of $R_j-\{z\}$. Since in this case $D_j$ is also an 
    $(R_j: W \cap V(R_j))$-dominating set, we have
     $\gamma(R_j: W \cap V(R_j)) \leq \gamma (R_j-\{z\}) =   |D_j|  < 2r_j/5$  by noting 
     that $r_j=4$ or 7.  
     
Now assume that all vertices of  $R_j$  are  in $W$. Let  $z$ be an
arbitrary vertex of $R_j$ and $D_j$ a minimum dominating set
of $R_j - \{z\}$. Let $A(D_j)$ be the set of arcs $xx'$ of $G$ 
such that $x \in D_j$ and $x'$ is the unique neighbour of $x$ in $S$. 
It is not hard to verify that $|D_j|  < 2r_j/5$.
Let $u$ be a neighbour of $z$ in $R_j$. Since $\d(R_j) = 2$ (see Figure \ref{fig:a}), 
we can choose a neighbour $y$ of $u$ in $G$ other than $u'$ and $z$ (so that $y$ 
is in $R_j$), where $u'$ is the unique neighbour of $u$ in $S$. 
In forming $A(S)$ by (\ref{eq:as}), we choose $A(u)$ to be
$\{u' u, uy, u' v\}$, where $v$ is a neighbour of $u'$ other than $u$. 
Similar to what we did in the proof of Theorem \ref{th:dom}, when necessary 
we may modify $A(S)$ to obtain a new set $A(S)$
(also denoted by $A(S)$) such that every arc of $G$ with tail $z$ is dominated in $X(G)$ 
by either $uy$ or an arc in $A(z')$, where $z'$ is the unique neighbour of $z$ in $S$. 
It can be verified that we can always choose an appropriate pair of vertices $z, u$ 
such that this happens. In this way we ensure that all arcs 
emanating from vertices of $R_j$ are dominated by $A(S) \cup A(D_j)$ in $X(G)$.

With $D_j$ as above we now set $D := \cup_{j=1}^s D_j$. As in (\ref{eq:D}), 
choose a set of arcs of $G$ by setting $A(D_j) = \{xx': x \in D_j\}$, where $x'$ is a neighbour of $x$ 
in $S$. (As seen in the previous paragraph, such a set $A(D_j)$ is unique 
when $R_j$ is isomorphic to some graph in $\cal{A}$ and $V(R_j) \subseteq W$.)
Set $A(D) = \cup_{j=1}^s A(D_j)$. Since $|D_j|  \le 2r_j/5$ for each $j$, 
we have $|A(D)| = |D|= \Sigma_{j=1}^s|D_j| \leq 2 |V(H')|/5 \leq 2 |V(G)-S|/5$. 
Similar to the proof of Theorem \ref{th:dom}, once can verify
that  $A(S)\cup A(D)$ is a dominating set of $X(G)$. Thus   
 \begin{eqnarray*}
 \gamma (X(G))  & \le &      |A(S)|+ |A(D)|  -1 \\ 
 &\leq&      3\gamma(G)+ \displaystyle{\frac{2 |V(G)-S|}{5}}-1  \\
           &\leq&    3\gamma(G)+    \displaystyle{\frac{2 ((\Delta(G)+1)\gamma (G)-\gamma(G))}{5}} -1  \nonumber  \\  
                &=&         \left(\displaystyle{\frac{2\Delta(G)}{5}}+3\right)\gamma(G)-1.  
 \end{eqnarray*}
 
(c) If  $\delta=4$, then  by  Theorem \ref{th:dom} and Lemma \ref{le1}(c),  
 \begin{eqnarray*}
 \gamma (X(G)) &\leq&  3\gamma(G)+ \gamma (H: V_{3}(H)) -1   \\
                 &\leq&    3\gamma(G)+    \displaystyle{\frac{ 3(n-\gamma(G)+2)}{8}}  -1   \\
           &\leq&    3\gamma(G)+    \displaystyle{\frac{3 ((\Delta(G)+1)\gamma (G)-\gamma(G)+2)}{8}}  -1 \nonumber  \\  
                       &=&  \left(\displaystyle{\frac{3(\Delta(G)+2)}{8}}+3\right)\gamma(G)-1.  
 \end{eqnarray*}
 The bounds in (a) and (b) are attained by the complete graphs $K_3$ and $K_4$ respectively.
This completes the proof.   \qed
  \end{proof}


A graph is {\em claw-free} if it does not contain the complete bipartite graph $K_{1,3}$ as an induced subgraph. The following result significantly improves the upper bounds in Theorem \ref{th:deg34} for claw-free graphs, and it does not require the minimum degree to be 2, 3 or 4.

\begin{theorem}
\label{th:claw}
  Let $G$ be a claw-free graph with $\delta(G)\geq 2$. Then 
  \be
  \label{eq:claw}
  \gamma (X(G))\leq 4\gamma(G)
  \ee
   and this bound is attainable.
\end{theorem}

 \begin{proof}  
 Let $S$ be a $\gamma(G)$-set of $G$. We prove (\ref{eq:claw}) by constructing a dominating set of $X(G)$ with size at most $4\gamma(G)$. 
 
Since $G$ is claw-free, for each $x \in V(G)$, the induced subgraph $G[N(x)]$ has independence number, and hence domination number, at most 2.
Thus, for each $x\in S$, we can choose a dominating set of $G[N(x)]$ with size $2$, say, $D_x := \{x_1, x_2\}$, where $x_1, x_2 \in N(x)$. We then associate      $x$ with a set  $A(x)$ of four arcs of $G$ in the following way.
   If $\deg(x)=2$, then let $A(x) =\{xx_1,xx_2,x_1x,x_2x\}$.   If $\deg(x)\geq 3$,  then
	let $x_3 \in N(x) - D_x$.
	Since  $D_x$ is a   dominating set  of  $G[N(x)]$, 
	  $x_3$ 
	is adjacent to at least one of $x_1$ and $x_2$.  We may assume without loss of generality that 
   $x_3$  is adjacent to $x_1$, and then we set $A(x) =\{xx_1, x_1x_3, x_3x, x_2x\}$. Define 
  $$
  A(S) := \cup_{x\in S} A(x).
  $$ 

We now show that $A(S)$ is a dominating set of $X(G)$. To this end it suffices to show that any arc $ab$  of $G$ outside $A(S)$ is dominated in $X(G)$ by at least one arc in $A(S)$. In fact, if $a\in S$, say, $a = x \in S$, then either $ab =xx_1\in A(x)$, or $b\ne x_1$  and  $ab$ is dominated by $x_1x_3 \in A(x)$. 

Suppose then that $a\notin S$. Since $S$ is a dominating set of $G$, $a$ has at least one neighbour in $S$, say, $x$. Consider the case $b=x$ first. In this case, if $a\in \{x_2,  x_3\}$ then $ab\in  A(x)$; if $a=x_1$ then $ab=x_1x$ is dominated by $x_3x$; and if $a\notin \{x_1, x_2,  x_3\}$ then one of  $x_1x_3$ and $x_2x$ dominates $ab$ since $\{x_1, x_2\}$ is  a dominating set of $G[N(x)]$. 
In the case where $b\ne x$, if $a\ne x_1$ then $ab$ is dominated by 
 $xx_1$, and if $a=x_1$ then either $ab=x_1x_3 \in A(S)$ or $ab$ is dominated by $x_3x$.
 
So far we have proved that every arc of $G$ outside $A(S)$ is dominated by $A(S)$. Therefore, $\gamma (X(G))\leq |A(S)| \leq 4|S|=4\gamma(G)$ and (\ref{eq:claw}) is established.
  
To show that (\ref{eq:claw}) is sharp, let $G^*$ be obtained from two vertex-disjoint complete graphs $K_s$ and $K_t$ of orders $s, t \ge 2$ respectively by identifying a vertex of $K_s$ with a vertex of  $K_t$. Since  $\gamma (X(H))\geq 3$ for  any graph $H$ with $\d(H) \ge 2$   \cite[Theorem 7]{KXZ}, we have $\g(X(G^*)) \ge 3$. However, no three arcs of  $G^*$ dominate all other arcs of $G^*$. Therefore, $\g(X(G^*)) = 4 = 4 \g(G^*)$.
   \end{proof}     \qed
 
Since line graphs are claw-free, the bound (\ref{eq:claw}) holds in particular when $G$ is the line graph of some graph of minimum degree at least two.

\vskip 1pc 
{\bf Acknowledgement}~~
Guangjun Xu was supported by the MIFRS and SFS scholarships of the University 
of Melbourne.

\end{document}